\documentclass[11pt,reqno]{amsart}

\usepackage{amsmath, amsfonts, amsthm, amssymb}

\newtheorem{Theorem}{Theorem}[section]
\newtheorem{Lemma}{Lemma}[section]
\newtheorem{Proposition}{Proposition}[section]

\theoremstyle{definition}
\newtheorem{Definition}{Definition}[section]

\theoremstyle{remark}

\numberwithin{equation}{section}

\renewcommand{\u}{{\bf u}}

\newcommand{\R}{{\mathbb R}}
\newcommand{\Dv}{{\rm div}}

\def\w{w}

\def\f{\frac}
\renewcommand{\O}{\Omega}

\def\D{\Delta }

\def\hf1{^\f{1}{1-\xi^2}}

\def\be{\begin{equation}}
\def\en{\end{equation}}
\def\bs{\begin{split}}
\def\es{\end{split}}

\renewcommand{\d}{{\bf d}}

\renewcommand{\v}{{\bf v}}

\author{Cheng Yu}
\address{Department of Mathematics, University of Pittsburgh,
                           Pittsburgh, PA 15260.}
\email{chy39@pitt.edu}
\title[Global well-posedness for 2D Navier-Stokes-Vlasov Equations]
{Global well-posedness for the two dimensional Navier-Stokes-Vlasov Equations}


\date{\today}

\begin{document}
\begin{abstract}
The global well-posedness for the incompressible Navier-Stokes-Vlasov equations in two spatial dimensions is established by a priori estimates, the characteristic method and the semigroup analysis.
\end{abstract}

\maketitle

\section{Introduction}
The objective of this paper is to establish the global well-posedness for the two dimensional Navier-Stokes-Vlasov equations:
\begin{equation}
\label{NSV}
\begin{split}
&\partial_t\u+\u\cdot\nabla\u+\nabla P-\mu\D\u=-\int_{\R^2}(\u-\v)f\,d\v,
\\&\Dv\u=0,
\\&\partial_t f+\v\cdot\nabla_{x}f+\Dv_{\v}((\u-\v)f)=0
\end{split}
\end{equation}
in $(0,T)\times\R^2\times\R^2$, with the following initial data
\begin{equation}
\label{initialdata}
\u(x,0)=\u_0(x),\;\;\;\;f(x,\v,0)=f_0(x,\v),
\end{equation}
where $\u$ is the velocity of the fluid, $P$ is the pressure, $\mu$ is the kinematic viscosity of the fluid.
Without loss of generality,
we take $\mu=1$ throughout the paper. The distribution function $f(t,x,\v)$
depends on the time $t\in[0,T]$, the physical position $x\in\R^2$ and the velocity of particle $\v\in\R^2$. The number of particles enclosed at $t\geq 0$ and location $x\in\R^2$ in the volume element $d\v$ is given by $f(t,x,\v)\,d\v$. We refer the readers to \cite{CP,H,O,RM,W} for more physical
background and discussion of the Navier-Stokes-Vlasov equations and related problems.

There have been many mathematical studies on the Navier-Stokes-Vlasov equations and related problems. The global existence for the Stokes-Vlasov system in a bounded domain was established in \cite{H}. The existence theorem for weak solutions has been extended in \cite{BDGM}, where the author did not neglect the convection term and considered the Navier-Stokes-Vlasov equations within a periodic domain.
The weak solution of the Navier-Stokes-Vlasov-Poisson system with corresponding boundary value problem was obtained in \cite{AKS}. The global existence of smooth solutions with small data for the Navier-Stokes-Vlasov-Fokker-Planck equations was obtained in \cite{GHMZ}.
More recently, the existence of global weak solutions with large data to the Navier-Stokes-Vlasov equations in a bounded domain was established in \cite{YU}.

However, there is no existence theory available for the Navier-Stokes-Vlasov equations with initial data in the whole space.
Compared with \cite{BDGM,YU}, the new difficulty is the loss of compactness of $\int_{\R^2}f\,d\v$ and $\int_{\R^2}\v f\,d\v$ in the whole space. Since the methods in \cite{BDGM,YU} do not work here, mathematical analysis for this problem is challenging and requires new ideas and techniques.
In this paper we shall study the initial value problem \eqref{NSV}-\eqref{initialdata} and establish the global well-posedness with large initial data. To achieve our goal, we will derive a priori estimates and use fixed point arguments. Partially motivated by the work of \cite{CK}, we will use the semigroup analysis to establish the iteration of $(\u,f)$ for using the fixed point theorem. To overcome the difficulty of the estimates of distribution function $f$, we adopt the idea as in \cite{BDGM}, and apply the characteristic method to the Vlasov equation, then the existence and uniqueness of the solutions for the Vlasov equation follows when $\u$ is continuous with respect to time $t$. The continuous dependence of the solution $f$ on $\u$ is also established. We also use Lemma \ref{L1} to deal with the distribution function and its coupling and interaction with fluid variables. With the above a priori estimates, continuous dependence, and semigroup analysis, we can apply the fixed point theorem to obtain the existence and uniqueness of strong solutions to problem \eqref{NSV}-\eqref{initialdata}. Further regularity of $(\u,f)$ can be deduced from the strong solution, thus the global well-posedness can be established.

In what follows, we denote $$m_kf=\int_{\R^2}|\v|^kf\, d\v,\;\;\text{ and } \;\;M_kf=\int_{\R^2}\int_{\R^2}|\v|^kf\,d\v dx,$$
$$\rho=\int_{\R^2}f\,d\v,\;\;\;\; j=\int_{\R^2}\v f\,d\v.$$
 It is easy to see that
\begin{equation*}
M_kf=\int_{\R^2}m_kf\,dx.
\end{equation*}
Here we state the following lemma due to \cite{H}:
\begin{Lemma}
\label{L1}
Suppose that $(\u,f)$ is a smooth solution to
\eqref{NSV}-\eqref{initialdata}. If $f_0\in L^{p}$ for some $p>1$, we have
\begin{equation*}
\|f(t,x;\v)\|_{L^p}\leq C(T)\|f_0\|_{L^{p}}, \text{ for any } t\geq0;
\end{equation*}
and if $|\v|^kf_0\in L^{1}(\R^2\times\R^2),$ then we have
\begin{equation*}
\int_{\R^2\times\R^2}|\v|^kf\,d\v dx\leq C(T)\left(\left(\int_{\R^2\times\R^2}|\v|^kf_0\,d\v dx\right)^{\frac{1}{2+k}}+(||f_0||_{L^{\infty}}+1)\|\u\|_{L^{r}(0,T;L^{2+k})}\right)^{2+k}.
\end{equation*}
\end{Lemma}

Our first main result reads as follows.
\begin{Theorem}
\label{T}
If $\u_0\in W^{1,2}(\R^2)$ is a divergence-free vector, $f_0\in C^{1}(\R^2\times\R^2)$, $M_6f_0<\infty$,  then there exists a unique strong solution $(\u,f)$ to \eqref{NSV}-\eqref{initialdata} for any $T>0$.
\end{Theorem}
The strong solution to system \eqref{NSV}-\eqref{initialdata} is defined as follows:
\begin{Definition}
A pair $(\u,f)$ is called a strong solution to the system \eqref{NSV}-\eqref{initialdata} if
\begin{itemize}
\item  $\u\in C(0,T;W^{1,2}(\R^2))\cap L^{2}(0,T;W^{2,2}(\R^2))$;
\item $f(t,x,\v)\geq 0, \text{ for any } (t,x,\v)\in (0,T)\times\R^2\times\R^2$;
\item $f\in C^{1}(0,T;W^{1,2}(\R^2\times\R^2))$;
\item $f|\v|^2\in L^{\infty}(0,T;L^{1}(\R^2\times\R^2))$.

 \end{itemize}
\end{Definition}
Based on Theorem \ref{T}, we can differentiate the system and apply similar arguments to obtain the further result:
\begin{Theorem}
\label{T2}
If $\u_0\in W^{m+1,2}(\R^2)\,\, m>0 \text{ integer},$ is a divergence-free vector, $f_0\in C^{1}(\R^2\times\R^2)$, $M_6f_0<\infty$,  then the solution satisfies
\begin{equation*}
\begin{split}
&\u\in C(0,T;W^{m+1,2}(\R^2))\cap L^2(0,T;W^{m+2,2}(\R\times\R^2));
\\&f(t,x,\v)\geq 0, \text{ for any } (t,x,\v)\in (0,T)\times\R^2\times\R^2;
\\&f\in C^{1}(0,T;W^{m+1,2}(\R^2\times\R^2)).
\end{split}
\end{equation*}
\end{Theorem}

\bigskip

\section{a Priori Estimates}

The aim of this section is to obtain some a priori estimates. We start with deriving the energy inequality.
Multiplying by $\u$ the both sides of the first equation in \eqref{NSV}, integrating over $\R^2$ and by parts,
we have
\begin{equation}
\label{energyofns}
\frac{d}{dt}\int_{\R^2}\frac{1}{2}|\u|^2\,dx+\int_{\R^2}|\nabla\u|^2\,dx=-\int_{\R^2}\int_{\R^2}f(\u-\v)\u\, d\v\, dx.
\end{equation}
Multiplying by $(1+\frac{1}{2}|\v|^2)$ the both sides of the third equation in \eqref{NSV} and integrating over $\R^2$ and by parts, one obtains that
\begin{equation}
\begin{split}
\label{energyofvlasov}
&\frac{d}{dt}\int_{\R^2\times\R^2}f(1+|\v|^2)\,d\v dx+\int_{\R^2}\int_{\R^2}f|\u-\v|^2d\v \,dx
\\&=\int_{\R^2}\int_{\R^2}f(\u-\v)\u\, d\v\, dx.
\end{split}
\end{equation}
Using \eqref{energyofns}-\eqref{energyofvlasov}, one obtains
\begin{equation}
\label{energyofnsv}
\begin{split}
&\frac{d}{dt}\left(\int_{\R^2}|\u|^2\,dx+\int_{\R^2\times\R^2}f(1+|\v|^2)\,\d\v dx\right)+2\int_{\R^2}|\nabla\u|^2dx\\&+2\int_{\R^2\times\R^2}f|\u-\v|^2\,d\v dx = 0.
\end{split}
\end{equation}
Taking the curl of the first equation in \eqref{NSV}, we obtain
\begin{equation}
\label{curlequality}
\partial_t\omega+\u\cdot\nabla\omega-\Delta\omega=\nabla^{T}\cdot(-\u\rho+j),
\end{equation}
where $\omega=\text{curl}\u$.
Multiplying by $\omega$ the both sides of \eqref{curlequality} and integrating, we obtain, after integration by parts,
\begin{equation}
\label{curlenergy}
\frac{d}{dt}\|\omega\|_{L^2}^2+\|\nabla\omega\|_{L^2}^2\leq C\left(\|\u\|_{L^4}^4+\|\rho\|_{L^4}^4+\|j\|_{L^2}^2\right).
\end{equation}

On the other hand, by \eqref{energyofnsv}, we have
$$\u\in L^{2}(0,T;W^{1,2}(\R^2)),$$
which implies that
\begin{equation}
\label{Lpregularityofu}
\u \in L^2(0,T;L^p(\R^2)),\;\;\text{ for any }p\geq 1.
\end{equation}
Applying Lemma \ref{L1}, the fact $M_6f_0<\infty,$ and \eqref{Lpregularityofu}, we obtain
 $$M_6f<\infty.$$
Applying Lemma 1 as in \cite{BDGM} in the two-dimensional space, we can control $\|\rho\|_{L^4}$ and $\|j\|_{L^2}$
by $M_6f$. Thus, we obtain that
\begin{equation}
\sup_{0\leq t\leq T}\|\omega\|_{L^2}^2+\int_0^T\|\nabla\omega\|_{L^2}^2dt\leq C(T),
\end{equation}
which implies that
\begin{equation*}
\u \in L^{\infty}(0,T;W^{1,2}(\R^2)) \cap L^2(0,T;W^{2,2}(\R^2)).
\end{equation*}
Multiplying by $\u_t$ the both sides of the first equation in \eqref{NSV}, using integration by parts, we obtain
\begin{equation}
\label{estimate of velocity on time}
\begin{split}
&\frac{\partial}{\partial t}\int_{\R^2}|\nabla\u|^2\,dx+\int_{\R^2}|\u_t|^2\,dx
\\&\leq C\left(\|\u\|_{L^4}^4+\|\rho\|_{L^4}^4+\|j\|_{L^2}^2\right)+\int_{\R^2}|\nabla\u|^2\,dx.
\end{split}
\end{equation}
Applying Gronwall's inequality, one obtains that
\begin{equation*}
\u_t\in L^{2}(0,T;L^2(\R^2)),\;\;\;\text{ and }\;\;\u\in L^{\infty}(0,T;W^{1,2}(\R^2)).
\end{equation*}
Now, we rely on the following Lemma which is a very special case of interpolation theorem of Lions-Magenes. We refer the readers to \cite{T} for the proof of this lemma.
\begin{Lemma}
\label{L5.1}
Let $V\subset H\subset V'$ be three Hilbert spaces, $V'$ is a dual space of $V$. If a function $\u$ belong to $L^{2}(0,T;V)$ and its derivative $\u'$ belongs to $L^2(0,T;V')$ then $\u$ is almost everywhere equal to a function continuous from $[0,T]$ into $H$.
\end{Lemma}
Applying Lemma \ref{L5.1} with the following facts
\begin{equation*}
\frac{\partial \u}{\partial t} \in L^{2}(\O\times(0,T)),\text{ and } \u \in L^{\infty}(0,T;W^{1,2}(\R^2)) \cap L^2(0,T;W^{2,2}(\R^2)),
\end{equation*}
we conclude that $\u\in C(0,T;W^{1,2}(\R^2))$.\\\\
On the other hand, we can apply maximal principle to the Vlasov equation to obtain
$$\|f\|_{L^{\infty}(0,T;L^{\infty}(\R^2\times\R^2)\cap L^{1}(\R^2\times\R^2))}\leq \|f_0\|_{L^{\infty}(\R^2\times\R^2)\cap L^{1}(\R^2\times\R^2))}.$$
Thus we proved
\begin{Proposition}
\label{P1}
Let $(\u,f)$ be a solution of \eqref{NSV}-\eqref{initialdata} on $[0,T],$ with $\u_0\in W^{2,2}(\R^2)$ and $f_0\in L^{\infty}(\R^2\times\R^2)\cap L^{1}(\R^2\times\R^2))$, $M_6f_0\leq C<\infty,$ then we have the following regularity:
\begin{equation*}
\begin{split}
&\u\in C(0,T;W^{1,2}(\R^2))\cap L^2(0,T;W^{2,2}(\R^2));
\\&f\in L^{\infty}(0,T;L^{\infty}(\R^2\times\R^2)\cap L^{1}(\R^2\times\R^2)).
\end{split}
\end{equation*}
\end{Proposition}

\section{Local and Global Strong Solution}

\begin{Proposition}
\label{P2}
Let $\u_0\in W^{2,2}(\R^2)$ be a divergence-free vector, $f_0\in L^{\infty}(\R^2\times\R^2)\cap L^{1}(\R^2\times\R^2))$, $M_6f_0\leq C<\infty,$ then there exists a time $T_0>0$ depending on the initial data and a unique strong solution
\begin{equation*}
\begin{split}
&\u \in C(0,T_0,\mathbb{P}W^{1,2}(\R^2))\cap L^{2}(0,T_0,\mathbb{P}W^{2,2}(\R^2));
\\& f\in L^{\infty}(0,T_0, L^{\infty}(\R^2\times\R^2)\cap L^{1}(\R^2\times\R^2))
\end{split}
\end{equation*}
of \eqref{NSV} with the initial data $(\u_0,f_0),$ where $\mathbb{P}$ is the Leray-Hodge projector on divergence-free vector. In addition, if $f_0\in C^{1}(\R^2\times\R^2)$, then we have $f\in C^{1}(0,T_0;W^{1,2}(\R^2\times\R^2)).$
\end{Proposition}

\begin{proof}
We define $\|(\u,f)\|_{B}=\|\u\|_{X}+\|f\|_{Y},$
where $$X=L^{\infty}(0,T_0,\mathbb{P}W^{1,2}(\R^2))\cap L^{2}(0,T_0,\mathbb{P}W^{2,2}(\R^2)),$$
$$ \|\u\|_{X}=\|\u\|_{L^{\infty}(0,T_0, \mathbb{P}W^{1,2}(\R^2))}+\|\u\|_{ L^{2}(0,T_0, \mathbb{P}W^{2,2}(\R^2))};$$
and
$$Y= L^{\infty}(0,T_0, L^{\infty}(\R^2\times\R^2)\cap L^{1}(\R^2\times\R^2)),$$
$$\|f\|_{Y}=\|f\|_{L^{\infty}(0,T_0, L^{\infty}(\R^2\times\R^2)\cap L^{1}(\R^2\times\R^2))}.$$
Clearly the spaces $\,X, \, Y$ are Banach spaces, and thus $B$ is Banach space.

We let $U=(\u,f)$ in the Banach space $B$, define the operator $T(U)$ in $B$, as $T(U)=(\bar{\u},\,\bar{f}),$
where $\bar{\u},\,\bar{f}$ are given by
\begin{equation}
\label{semigroup equation}
\begin{split}
&\bar{\u}=e^{t\Delta}\u_0+\int_0^{t}e^{(t-s)\Delta}\mathbb{P}(\u\cdot\nabla\u)ds+\int_{0}^te^{(t-s)\Delta}\mathbb{P}(-\rho\u+j)ds,
\\&\bar{f}=N(\u,\v) \;\;\text{ where }\;\partial_t\bar{f}+\v\cdot\nabla\bar{f}+\Dv_{\v}((\u-\v)\bar{f})=0,\;\;\bar{f}(x,\v,0)=f_0(x,\v).
\end{split}
\end{equation}
We denote $$Q(\u,\w)=\int_0^te^{(t-s)\Delta}\mathbb{P}(\u\cdot\nabla\w),$$ which solves
\begin{equation*}
\partial_t Q-\Delta Q=\mathbb{P}(\u\cdot\nabla\w),\;\;\; Q(x,0)=0.
\end{equation*}
It is easy to obtain the following energy inequality,
\begin{equation*}
\frac{d}{dt}\|\Delta Q\|_{L^2}^2+\|\nabla\Delta Q\|_{L^2}^2\leq \|\nabla(\u\cdot\nabla\w)\|_{L^2}\|\nabla\Delta Q\|_{L^2}.
\end{equation*}
 Using Ladyzhenskaya inequality for the term involving $\nabla\u\cdot\nabla \w$ and the interpolation inequality for the term involving $\u\cdot\nabla(\nabla\w)$, one obtains that
\begin{equation*}
\sup_{0\leq t\leq T}\|\Delta Q\|_{L^2}^2+\int_0^{T_0}\|\nabla\Delta Q\|_{L^2}^2\,dt\leq CT_0\|\u\|_X^2\|\w\|_X^2.
\end{equation*}

We denote $$L:=\int_0^te^{(t-s)\Delta}\mathbb{P}(-\rho\u+j)ds,$$ which solves
\begin{equation*}
\partial_tL-\Delta L=\mathbb{P}(-\rho\u+j),\;\;\; L(x,0)=0.
\end{equation*}
Multiplying $\Delta L$ the both sides of above equation, and using integration by parts, we get
\begin{equation}
\label{secondenergy}
\begin{split}
\sup_{0\leq t\leq T}\|\nabla L\|_{L^2}^2+2\int_0^{T_0}\|\Delta L\|_{L^2}^2dt&\leq \int_0^{T_0}\|\Delta L\|_{L^2}^2ds\\&+\int_0^{T_0}\|\rho\u\|_{L^2}^2\,ds+\int_0^{T_0}\|j\|_{L^2}^2\,ds.
\end{split}
\end{equation}
By the Cauchy-Schwartz inequality, we have
\begin{equation*}
\sup_{0\leq t\leq T}\|\nabla L\|_{L^2}^2+\int_0^{T_0}\|\Delta L\|_{L^2}^2dt\leq \int_0^{T_0}\left(\|\rho\|_{L^4}^2+\|\u\|_{L^4}^2+\|j\|_{L^2}^2\right)\,ds.
\end{equation*}
Using Lemma 1 as in \cite{BDGM} and Lemma \ref{L1}, we can control $\|\rho\|_{L^4}$ as follows
\begin{equation*}
\|\rho\|_{L^4}\leq M_6f\leq C(M_6f_0+\|\u\|_X)^8.
\end{equation*}
Similarly, we can control the term $\|j\|_{L^2}$. Thus, we have the following estimate:
\begin{equation*}
\sup_{0\leq t\leq T}\|\nabla L\|_{L^2}^2+\int_0^{T_0}\|\Delta L\|_{L^2}^2dt\leq C(1+\|\u\|_X)^8.
\end{equation*}
Integrating the third equation in \eqref{NSV} with respect to $x$ and $\v$, we have
\begin{equation}
\label{conservation}
\frac{d}{dt}\int_{\R^2}\int_{\R^2}f dx\,d\v=0.
\end{equation}
Applying the maximum principle to the third equation in \eqref{NSV}, one obtains that
\begin{equation}
\label{maximum}
\|f\|_{L^{\infty}(\R^2\times\R^2)}\leq C_T\|f_0\|_{L^{\infty}(\R^2\times\R^2)}
\end{equation}
for all $t\in [0,T],$ if $f_0\in L^{\infty}(\R^2\times\R^2).$  From \eqref{conservation} and \eqref{maximum}, $f=N(\u,\v)$ satisfies the following estimate
\begin{equation}
\label{theboundoff}
\|f\|_{L^{\infty}(0,T,L^{\infty}(\R^2\times\R^2)\cap L^1(\R^2\times\R^2))}\leq C_{T_0}\|f_0\|_{L^{\infty}(\R^2\times\R^2)},
\end{equation}
which means
\begin{equation}
\label{theboundofN}
\|N(\u,\v)\|_{L^{\infty}(0,T_0,L^{\infty}(\R^2\times\R^2)\cap L^{1}(\R^2\times\R^2))}\leq C_{T_0}\|f_0\|_{L^{\infty}(\R^2\times\R^2)}.
\end{equation}
Following the same argument of \cite{BDGM}, we let $g=e^{-2t}f$ satisfy the following transport equation
\begin{equation*}
\partial_t g+\v\cdot\nabla g+(\u-\v)\cdot \nabla_{\v}g=0.
\end{equation*}
The above equation can be written by characteristics method as follows,
\begin{equation*}
\begin{split}
&\frac{dx}{dt}=\v(t),
\\&\frac{d\v}{dt}=\u(t,x(t))-\v(t),
\end{split}
\end{equation*}
with the initial data
$$x(0)=x\;\;\;\text{ and }\,\,\,\,\v(0)=\v,$$
and set $\chi(t,x,\v)=(x(t),\v(t))$ for any $(t,x,\v)$.

 Applying the fact $\u\in C(0,T;W^{1,2}(\R^2))$ and the classical theory of Ordinary Differential Equations, we obtain the unique solution
\begin{equation}
\label{solutionofode}
f(t,x,\v)=e^{2t}f_0(\chi(t,x,\v)),\;\;\;\text{ for any}\,\, (t,x,\v).
\end{equation}
Thus, we conclude that $$f\in C^1(0,T;C^{0}(\R^2\times\R^2))\;\;\;\text{ if }\;\;f_0\in C^{1}(\R^2\times\R^2).$$
Define $$f_1=N(\u_1,\v),\; f_2=N(\u_2,\v),$$ and $$\chi_1=(x(\u_1),\v),\;\chi_2=(x(\u_2),\v).$$

Notice that $f_1-f_2$ can be controlled as follows
\begin{equation}
\label{estimateoff}
\|f_1-f_2\|_Y\leq C(T)\|\chi_1-\chi_2\|_{L^{\infty}(\R^2\times\R^2)}.
\end{equation}
By the definition of $\chi=(x,\v)$, we have the following estimate
\begin{equation}
\begin{split}
\label{continuousdependent}
&\|(\chi_1-\chi_2)(t)\|_{L^{\infty}(\R^2\times\R^2)} \\&\leq C\left(\int_0^t\|(\u_1-\u_2)(s)\|_{L^{\infty}(\R^2)}ds+\int_0^t(1+\|\u(s)\|_X)\|(\chi_1-\chi_2)(s)\|_{L^{\infty}(\R^2\times\R^2)}ds\right).
\\&\leq C\left(\int_0^t\|(\u_1-\u_2)(s)\|_{L^{\infty}(\R^2)}ds+\int_0^t\|(\chi_1-\chi_2)(s)\|_{L^{\infty}(\R^2\times\R^2)}ds\right).
\end{split}
\end{equation}
Thus, for any $t$, applying Gronwall's inequality to \eqref{continuousdependent}, we have
\begin{equation*}
\|\chi_1-\chi_2\|_Y\leq C\varepsilon \|\u_1-\u_2\|_{X}.
\end{equation*}
This, with the help of \eqref{estimateoff}, implies that
\begin{equation*}
\|f_1-f_2\|_Y\leq C\varepsilon \|\u_1-\u_2\|_X.
\end{equation*}
Thus
\begin{equation}
\label{estimateoff2}
\|N(\u_1,\v)-N(\u_2,\v)\|_Y\leq C\varepsilon \|\u_1-\u_2\|_X.
\end{equation}

Set $U^{0}=(\u_0,f_0)$ and define the iteration $U^{n+1}=T(U^n)\,\text{ for }\,n=0,1,2,...$. It is easy to see that the sequence $U^n$ is bounded in $B$ and converges if we choose $\varepsilon$ small enough.
If there exist $A,\,D$ and $\varepsilon$ such that
\begin{equation*}
\|\u^n\|_X\leq A,\;\;\;\|f^n\|_{Y}\leq D,
\end{equation*}
then, by induction and \eqref{semigroup equation}, we have
\begin{equation*}
\|\u^{n+1}\|_X\leq A_0+\varepsilon A^2+\varepsilon C(1+A)^8,\;\;\|f^{n+1}\|_Y\leq C_{T_0}D_0.
\end{equation*}
We can choose $\varepsilon$ small enough, such that $\varepsilon(1+A)^8+\varepsilon A^2+A_0\leq A$,
and choose $D$ such that $C_{T_0}D_0\leq D.$
Thus, we conclude that the sequence is bounded in $B$, then we can obtain the convergence of $\u^n$ in $X$, $f^n$ in $Y$.

\end{proof}

By Proposition \ref{P2}, there exists a strong solution on a short time interval $[0,T_0].$ For any given $T>0,$ we consider the maximal interval of the existence, $T_1=\sup T_0\leq T,$ such that the solution is strong on $[0,T_0].$ The main goal is to prove that $T_1$ can be taken to be equal to $+\infty.$
For any given $T_0>0,$ there exists a constant $K>0$ such that
\begin{equation*}
\|f\|_Y\leq \frac{K}{C_{T_0}},
\end{equation*}
which implies $\|f(T_0,x,\v)\|_Y\leq K.$
Using a priori estimates in Section 2, and applying Proposition \ref{P2}, the strong solution can be extended to $[0,T_0+T^*]$ for a small number $T^{*}>0$. One can then repeat the argument many times and obtain the existence and uniqueness on the whole real line. Thus we proved Theorem \ref{T}.
\\\\

To prove Theorem \ref{T2}, the further regularity $(\u,f)$ can be deduced from the regularity from Theorem \ref{T}. We can differentiate the equation \eqref{NSV} and apply similar arguments, Theorem \ref{T2} follows.

\end{document}